\documentclass[12pt]{article}

\usepackage{amssymb,amsthm,amsmath}
\usepackage[usenames,dvipsnames]{color}
\usepackage{pgf,tikz}
\usetikzlibrary{arrows}
\usetikzlibrary{snakes}
\usepackage{graphicx}
\usepackage{caption}
\usepackage{subcaption}
\usepackage{caption}
\usepackage{longtable}
\allowdisplaybreaks

\usepackage{cite}

\newcommand{\hwp}{\mathrm{HWP}}

\newcommand{\comment}[1]{}
\definecolor{teal}{RGB}{0,128,128}
\definecolor{darkpurple}{RGB}{128,0,128}
\definecolor{unbred}{RGB}{204,0,0}

\usepackage[titletoc,toc,title]{appendix}

\newtheorem{theorem}{Theorem}[section]

\newtheorem{lemma}[theorem]{Lemma}
\newtheorem{corollary}[theorem]{Corollary}

\theoremstyle{definition}

\def \Z {\mathbb Z}

\title {The Hamilton-Waterloo Problem with even cycle lengths}

\author{A.\ C.\ Burgess \footnotemark[1] 
\and 
P.\ Danziger \footnotemark[2] 
\and
T.\ Traetta \footnotemark[3]
}

\begin{document}

\maketitle

\footnotetext[1]{Department of Mathematics and Statistics, University of New Brunswick, 100 Tucker Park Rd., Saint John, NB,  E2L 4L5, Canada}
\footnotetext[2]{Department of Mathematics, Ryerson University, 350 Victoria St., Toronto, ON,  M5B 2K3, Canada}
\footnotetext[3]{DICATAM, Universit\`{a} degli Studi di Brescia, Via Branze 43, 25123 Brescia, Italy}

\begin{abstract}
The {\em Hamilton-Waterloo Problem} $\hwp(v;m,n;\alpha,\beta)$ asks for a 2-factorization of the complete graph $K_v$ or $K_v-I$, the complete graph with the edges of a 1-factor removed, into $\alpha$ $C_m$-factors and $\beta$ $C_n$-factors, where $3 \leq m < n$.   
In the case that $m$ and $n$ are both even, the problem has been solved except possibly when $1 \in \{\alpha,\beta\}$ or when $\alpha$ and $\beta$ are both odd, in which case necessarily $v \equiv 2 \pmod{4}$.  In this paper, we develop a new construction that creates factorizations with larger cycles from existing factorizations under certain conditions.  This construction enables us to show that there is a solution to $\hwp(v;2m,2n;\alpha,\beta)$ for odd $\alpha$ and $\beta$ whenever the obvious necessary conditions hold, except possibly if $\beta=1$; $\beta=3$ and $\gcd(m,n)=1$; $\alpha=1$; or $v=2mn/\gcd(m,n)$. This result almost completely settles the existence problem for even cycles, other than the possible exceptions noted above.
\end{abstract}


\section{Introduction}

We assume that the reader is familiar with the basic concepts of graph theory.  In particular, we use $V(G)$ and $E(G)$ to denote the vertex set and edge set, respectively, of the graph $G$.  
A {\it $k$-factor} of $G$ is a $k$-regular spanning subgraph of $G$.  Thus a {\em 1-factor} of $G$ (also called a {\em perfect matching}) is a collection of independent edges whose end-vertices partition $V(G)$, and a {\em 2-factor} of $G$ is a collection of vertex-disjoint cycles in $G$ whose vertex sets partition $V(G)$.  If the cycles in a given 2-factor all have the same length, we say that the 2-factor is {\em uniform}.  We will use the notation $C_{\ell}$ to denote a cycle of length $\ell$, and refer to a uniform 2-factor whose cycles have length $\ell$ as a {\em $C_{\ell}$-factor}.  

If $G_1, G_2, \ldots, G_r$ are subgraphs of $G$ whose edge sets partition $E(G)$, then we speak of a {\em decomposition} of $G$ into its subgraphs $G_1, \ldots, G_r$, and write $G=G_1 \oplus G_2 \oplus \ldots \oplus G_r$.  In particular, a {\em 2-factorization} of $G$ is a decomposition of $G$ into 2-factors.  If $\mathcal{F}=\{F_1, \ldots, F_t\}$ is a set of 2-factors of $G$, then we refer to a 2-factorization in which every factor is isomorphic to an element of $\mathcal{F}$ as an $\mathcal{F}$-factorization.  If $\mathcal{F}=\{F\}$, then we speak of an $F$-factorization; if, moreover, $F$ is a $C_{\ell}$-factor, then we refer to a $C_{\ell}$-factorization.  

We are particularly interested in 2-factorizations of $K_v$, the complete graph of order $v$.  Note that if $v$ is even, then $K_v$ has no 2-factorization, as its vertices have odd valency.  Thus we define $K_v^*$ to denote $K_v$ if $v$ is odd and $K_v-I$, the complete graph with the edges of a 1-factor $I$ removed, if $v$ is even.  The question of whether $K_v^*$ admits a 2-factorization in which each 2-factor is isomorphic to $F$ is known as the {\em Oberwolfach Problem} $\mathrm{OP}(F)$, and has been the subject of much study.  The Oberwolfach Problem has been solved in the case that $F$ is uniform~\cite{ASSW, Hoffman Schellenberg 91}, bipartite~\cite{BryantDanziger, Haggkvist 85} or contains exactly two components~\cite{Traetta 13}.  The solution of the Oberwolfach Problem for uniform factors will be useful to us later, so we state it here for future reference.
\begin{theorem}[\cite{ASSW, Hoffman Schellenberg 91}] \label{uniform OP}
Let $v, \ell \geq 3$ be integers.  There is a $C_{\ell}$-factorization of $K_v^*$ if and only if $\ell \mid v$ and $(v,\ell) \notin \{(6,3), (12,3)\}$.
\end{theorem}
Complete solutions to the Oberwolfach Problem are also known for certain infinite families of orders~\cite{ABHMS, Bryant Schar 09} and asymptotic solutions can be found in~\cite{DukesLing, GJKKO}; however, the problem is still open in general.  

The Oberwolfach Problem has been extended to finding 2-factorizations of regular graphs other than $K_v^*$, notably certain classes of lexicographic product.  We use $G[n]$ to denote the lexicographic product of the graph $G$ with the empty graph on $n$ vertices, so that $V(G[n]) = V(G) \times \mathbb{Z}_n$, with $(u,x)(v,y) \in E(G[n])$ if and only if $uv \in E(G)$ and $x,y \in \mathbb{Z}_n$.  Of particular note, $K_m[n]$ is the complete equipartite graph with $m$ parts of size $n$; the existence of uniform 2-factorizations of $K_m[n]$ was settled by Liu~\cite{Liu00, Liu03}.
\begin{theorem}[\cite{Liu00,Liu03}] \label{liu}
Let $\ell, m, n$ be positive integers with $\ell \geq 3$.  There is a $C_{\ell}$-factorization of $K_{m}[n]$ if and only if the following conditions are all satisfied:
\begin{enumerate}
\item $\ell \mid mn$;
\item $(m-1)n$ is even;
\item if $m=2$, then $\ell$ is even;
\item $(\ell,m,n) \notin \{(3,3,2), (3,6,2), (3,3,6), (6,2,6)\}$.
\end{enumerate}
\end{theorem}

A related question is the \emph{Hamilton-Waterloo Problem} $\mathrm{HWP}(G;F_1,F_2;\alpha,\beta)$.  Here, we seek a 2-factorization of the graph $G$ in which $\alpha$ 2-factors are isomorphic to $F_1$ and $\beta$ 2-factors are isomorphic to $F_2$.  In the case that $G=K_v^*$, we denote this problem by $\mathrm{HWP}(v;F_1,F_2;\alpha,\beta)$, while if $F_1$ and $F_2$ are uniform 2-factors, say $F_1$ is a $C_m$-factor and $F_2$ is a $C_n$-factor, we use the notation $\mathrm{HWP}(G;m,n;\alpha,\beta)$.  Thus, $\mathrm{HWP}(v;m,n;\alpha,\beta)$ asks whether $K_v^*$ has a 2-factorization into $\alpha$ $C_m$-factors and $\beta$ $C_n$-factors.  {We have the following obvious necessary conditions.  

\begin{theorem} \label{Necessary}
Let $G$ be a $2r$-regular graph, and let $F_1$ and $F_2$ be 2-factors of $G$.  If there is a solution to $\mathrm{HWP}(G;F_1, F_2; \alpha,\beta)$, then $\alpha, \beta \geq 0$ and $\alpha+\beta = r$.  In particular, there can be a solution to $\mathrm{HWP}(G;m,n;\alpha,\beta)$ only if $m$ and $n$ both divide $|V(G)|$, $\alpha, \beta \geq 0$ and $\alpha+\beta = r$.
\end{theorem}}

Note that when one of $\alpha$ or $\beta$ is 0, or when $m=n$,  $\mathrm{HWP}(v;m,n;\alpha,\beta)$ is equivalent to an instance of the uniform Oberwolfach Problem, so we will generally assume that $\alpha$ and $\beta$ are positive.  In addition, when considering $\hwp(G;m,n;\alpha,\beta)$, we will generally assume without loss of generality that $m < n$.

The Hamilton-Waterloo Problem $\hwp(v;m,n;\alpha,\beta)$ has been the subject of much recent study; see, for instance, the following papers, which have all appeared since 2013~\cite{AsplundEtAl, BonviciniBuratti, BryantDanzigerDean, BurattiDanziger, BDT3, BDT2, BDT1, KeranenOzkan, KeranenPastine, MerolaTraetta, OdabasiOzkan, WangCao, WangChenCao, WangLuCao}.  An asymptotic existence result is given in~\cite{GJKKO}. In the case that $m$, $n$ and $v$ are all odd, the current authors have solved $\hwp(v;m,n;\alpha,\beta)$ (recalling that $m < n$) except possibly if $\alpha=1$, $\beta \in \{1,3\}$ or $v=mn/\gcd(m,n)$~\cite{BDT2}.  When $m$ and $n$ have opposite parities, less is known.  The paper~\cite{BDT3} solves this problem when $m \mid n$, $v>6n>36m$ and $\beta \geq 3$; further results for cycle lengths of opposite parities can be found in~\cite{KeranenPastine2}.  The case $(m,n)=(3,4)$ is completely solved~\cite{BonviciniBuratti, DanzigerQuattrocchiStevens, OdabasiOzkan, WangChenCao}.  Other cases which have been considered include $(m,n) \in \{(3,v), (3,3s), (4,n), (8,n)\}$~\cite{AsplundEtAl, KeranenOzkan, LeiShen, OdabasiOzkan, WangCao}.

In this paper, we consider the Hamilton-Waterloo Problem $\mathrm{HWP}(v;m,n;\alpha,\beta)$ for even $m$ and $n$.  More generally, factorization into bipartite factors has been considered in~\cite{BryantDanziger, BryantDanzigerDean, Haggkvist 85}.

\begin{theorem}[\cite{BryantDanziger, Haggkvist 85}] \label{known results}
Let $v$ be a positive even integer and let $F_1$ and $F_2$ be bipartite 2-regular graphs of order $v$.  
\begin{enumerate}
\item If $v \equiv 0 \pmod{4}$, then there is a solution to $\hwp(v;F_1,F_2;\alpha,\beta)$ if and only if $\alpha+\beta=  \frac{v-2}{2}$, except possibly if $\alpha=1$ or $\beta=1$.
\item If $v \equiv 2 \pmod{4}$, then there is a solution to $\hwp(v;F_1,F_2;\alpha,\beta)$ whenever $\alpha+\beta= \frac{v-2}{2}$ and, in addition, $\alpha$ and $\beta$ are both even.
\end{enumerate}
\end{theorem}
In fact,~\cite{BryantDanziger} actually proves a more general result.

\begin{theorem}[\cite{BryantDanziger}]\label{BD}
Let $\mathcal{F}=\{F_1, F_2, \ldots, F_t\}$ be a collection of bipartite 2-regular graphs of order $v$ and let $\alpha_1, \alpha_2, \ldots, \alpha_t$ be nonnegative integers satisfying $\alpha_1+\alpha_2+\ldots+\alpha_t = \frac{v-2}{2}$.  If $\alpha_1 \geq 3$ is odd and $\alpha_i$ is even for each $i \in \{2,3,\ldots,t\}$, then $K_v$ admits an $\mathcal{F}$-factorization in which $\alpha_i$ factors are isomorphic to $F_i$, $i \in \{1,\ldots,t\}$.
\end{theorem}

Bryant, Danziger and Dean~\cite{BryantDanzigerDean} gave a complete solution to the Hamilton-Waterloo Problem with bipartite factors $F_1$ and $F_2$ in the case that $F_1$ is a {\em refinement} of $F_2$, i.e.\ $F_1$ can be obtained from $F_2$ by replacing each cycle of $F_2$ with a bipartite 2-regular graph on the same vertex set.
\begin{theorem}[\cite{BryantDanzigerDean}] \label{refinement}
Let $\alpha, \beta \geq 0$ and $v>0$ be integers with $v$ even, and let $F_1$ and $F_2$ be bipartite 2-regular graphs of order $v$ such that $F_1$ is a refinement of $F_2$.  There is a solution to $\hwp(v;F_1,F_2;\alpha,\beta)$ if and only if $\alpha+\beta=\frac{v-2}{2}$.
\end{theorem}

Note that a $C_m$-factor is a refinement of a $C_n$-factor if and only if $m \mid n$.  Thus, in the uniform case, Theorems~\ref{known results} and~\ref{refinement} yield the following:
\begin{theorem}[\cite{BryantDanziger, BryantDanzigerDean, Haggkvist 85}] \label{known uniform}
Let $v$ be a positive even integer, and let $n>m \geq 2$ and $\alpha, \beta \geq 0$ be integers.  There is a solution to $\hwp(v;2m,2n;\alpha,\beta)$ if and only if $2m$ and $2n$ are both divisors of $v$ and $\alpha+\beta=\frac{v-2}{2}$, except possibly when
\begin{enumerate}
\item $v \equiv 0$ (mod 4), $m \nmid n$, and $1 \in \{\alpha, \beta\}$;
\item $v \equiv 2$ (mod 4), $m \nmid n$, and $\alpha$ and $\beta$ are both odd.
\end{enumerate}
\end{theorem}

In this paper, we improve upon these results for uniform bipartite factors.  Since we assume the cycle lengths are even, we will henceforth consider $\hwp(G;2m,2n;\alpha,\beta)$.  In Section 2, we give a method for extending known solutions of $\hwp(C_m[n];m,n;\alpha,\beta)$ to obtain solutions of $\hwp(C_{2m}[n];2m,2n;\alpha,\beta)$.  This method is used along with other techniques in Section 3 to construct particular 2-factorizations of the lexicographic product of a cycle with an empty graph.  Finally, in Section 4, we present results on $\hwp(K_t[w];2m,2n;\alpha,\beta)$ and $\hwp(v;2m,2n;\alpha,\beta)$.  In particular, we give a near-complete solution to $\hwp(v;2m,2n;\alpha,\beta)$ when $\alpha$ and $\beta$ are odd, with possible exceptions remaining only when $\beta \in \{1,3\}$, $\alpha=1$ or $v=2mn/\gcd(m,n)$.  We also give some new sufficient conditions for the existence of a solution to $\hwp(v;2m,2n;1,\beta)$ when $v \equiv 0$ (mod 4).
\vspace*{2ex}

\section{Extending 2-factorizations}

As a step towards constructing solutions of $\hwp(v;2m,2n;\alpha,\beta)$, we will first consider the related problem of finding 2-factorizations of $C_{2m}[n]$.  It will be useful to view a graph $C_{wm}[n]$ as a type of Cayley graph, which we now define.  

Let $\Gamma$ be an additive group and $S \subseteq \Gamma \setminus \{0\}$.  The {\em Cayley graph} $\mathrm{Cay}(\Gamma,S)$ has vertex set $\Gamma$ and edge set $\{a (d+a) \mid a \in \Gamma, d \in S\}$.  Note that $d \in S$ and $-d \in -S$ generate the same edges, and so $\mathrm{Cay}(\Gamma,S) \cong \mathrm{Cay}(\Gamma,-S)$.  Hence $\mathrm{Cay}(\Gamma,S)$ is $|S \cup (-S)|$-regular.  

If $\Gamma=\Z_n$, then $\mathrm{Cay}(\Z_n,S)$ is a {\em circulant graph} with connection set $S$, denoted $\langle S \rangle_n$.  For future reference, we note the following result on 2-factorizations of Cayley graphs.

\begin{theorem}[\cite{Bermond}]\label{4-reg}
Every connected 4-regular Cayley graph admits a factorization into Hamilton cycles.
\end{theorem}

Let $M$ and $n$ be positive integers with $M \geq 3$.  It is easy to see that $C_{M}[n]\cong 
\mathrm{Cay}(\Z_{M}\times\Z_n,$ $\{1\}\times\Z_n)$.  We use $x_i$ to denote the vertex $(x,i) \in V(C_{M}[n]) = \Z_{M}\times\Z_n$.  An edge $e=x_i (x+1)_j \in E(C_{M}[n])$ has {\em difference} $j-i \in \Z_n$.  
Given a subset $S$ of $\Z_n$, we denote by $C_{M}[n,S]$ the subgraph of $C_{M}[n]$ induced by the edges whose differences are in $S$, i.e.\ $C_{M}[n,S] \cong \mathrm{Cay}(\Z_M \times \Z_n, \{1\} \times S)$.  When the value of $n$ is understood we will denote $C_{M}[n,S]$ by $C_{M}[S]$.
We say that an edge-disjoint set $\mathcal{F}$ of 2-factors of $C_{M}[n]$ {\em covers} $S$ if $\mathcal{F}$ is a 2-factorization of $C_{M}[S]$.

We will now show how to use existing 2-factorizations of $C_m[n,S]$ to construct a 2-factorization of 
$C_{wm}[n,S]$ in which the cycle lengths are multiplied by $w$.  The basic idea is as follows.  Given $w$ copies of $C_m[n,S]$ on vertex set $[0,m-1] \times \Z_n$ and a 2-factor of each, we replace edges of the form 
$(m-1)_i 0_j$ in each $C_m[n,S]$ with the corresponding edges joining one copy of $C_m[n,S]$ to the next, see the picture below.  The 2-factors are correspondingly spliced together to form a 2-factor of $C_{wm}[n,S]$.     

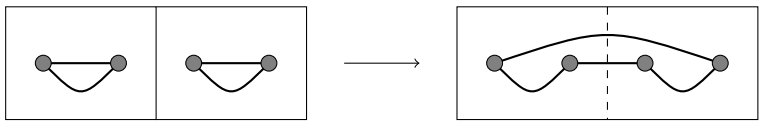
\begin{center}
\begin{tikzpicture}[x=1cm,y=1cm,scale=1]
\draw (-0.5,-0.75) -- (-0.5,0.75) -- (1.5,0.75) -- (1.5,-0.75) -- cycle;
\draw (1.5,-0.75) -- (3.5,-0.75) -- (3.5,0.75) -- (1.5,0.75);
\draw (5.5,-0.75) -- (9.5,-0.75) -- (9.5,0.75) -- (5.5,0.75) -- cycle;
\draw[dashed] (7.5,0.75) -- (7.5,-0.75);
\draw[thick] (0,0) -- (1,0);
\draw[thick] (0,0) .. controls (0.5,-0.5) .. (1,0);
\draw[thick] (2,0) -- (3,0);
\draw[thick] (2,0) .. controls (2.5,-0.5) .. (3,0);
\foreach \x in {0,1,2,3}{
\draw[fill=gray](\x,0) circle (3pt);
}
\draw[->] (4,0) -- (5,0);

\begin{scope}[shift={(6,0)}]
\draw[thick] (0,0) .. controls (0.5,-0.5) .. (1,0);
\draw[thick] (1,0) -- (2,0);
\draw[thick] (2,0) .. controls (2.5,-0.5) .. (3,0);
\draw[thick] (0,0) .. controls (1.5,0.5) .. (3,0);
\foreach \x in {0,1,2,3}{
\draw[fill=gray](\x,0) circle (3pt);
}

\end{scope}
\end{tikzpicture}
\end{center}

Although the main result of this section (Theorem~\ref{ASSW extension}) considers only the case that $w$ is a power of 2, we note that this method may be applied more generally, but care must be taken to avoid the creation of short cycles.  

We start by defining some notation.  Let $m,n$ and $w$ be positive integers, with $m\geq 3$ and $w\geq 2$.   
We take $C_m[n]$ to have vertex set $V(C_{m}[n])=[0,m-1]\times\Z_n$, so that $x_i y_j \in E(C_{m}[n])$ if and only if $x-y \in \{\pm 1, \pm (m-1)\}$.  Similarly, letting $P_{m}$ denote the path of length 
$m$, we take $V(P_{m}[n])=[0,m]\times\Z_n$, with $x_i y_j \in E(P_{m}[n])$ if and only if $x-y \in \{\pm 1\}$ (where this difference is computed in $\mathbb{Z}$).  Recalling that $C_{wm}[n]\cong \mathrm{Cay}(\Z_{wm}\times\Z_n, \{1\}\times\Z_n)$, in this section we view the vertex sets of $C_m[n]$ and 
$P_{m}[n]$ as subsets of $V(C_{wm}[n])$, so the intervals $[0,m-1]$ and $[0,m]$ are considered as subsets of $\Z_{wm}$.

Given a subgraph $H$ of $C_m[n]$, we denote by $H^*$ the subgraph of $C_{wm}[n]$ with vertices in $[0,m]\times \Z_n$
such that $E(H^*) = \mathcal{E}_1 \ \cup \ \mathcal{E}_2$ where
  \begin{align*}
    \mathcal{E}_1 &= E(H)\ \cap\ E(C_{wm}[n]),\\
    \mathcal{E}_2 &= \big\{(m-1, y_1)(m, y_2)\mid
    (0, y_2)(m-1, y_1)\in E(H)\big\},
  \end{align*}
As an example, the $23$-cycle $H$ in Figure \ref{H} is a subgraph of $C_{5}[7]$. The graph $H^*$, shown in Figure \ref{H*}, is the vertex disjoint union of two paths of length $5$ and one path of length $10$. 

\definecolor{ccqqqq}{rgb}{0.9,0.,0.}
\definecolor{qqzzff}{rgb}{0.,0.6,1.}
\definecolor{ffffff}{rgb}{1.,1.,1.}
\definecolor{cqcqcq}{rgb}{0.75,0.75,0.75}

\noindent
\begin{figure}[h]
\centering
\begin{minipage}[t]{.45\textwidth}
\begin{tikzpicture}[line cap=round,line join=round,>=triangle 45, x=.8cm,y=.9cm]
\draw[->,color=black] (1,0) -- (1,7.5);
\foreach \y in {0,1,2,3,4,5,6}
\draw[shift={(1,\y+1)},color=black] (2pt,0pt) -- (-2pt,0pt) node[left] {\footnotesize $\y$};
\draw[shift={(1,8)}, color=black] node {$\mathbb{Z}_{7}$};

\draw[->,color=black] (1,0) -- (6.5, 0);
\foreach \x in {0,1,2,3,4}
\draw[shift={(\x+2,0)},color=black] (0pt, 2pt) -- (0pt,-2pt) node[below, below] {\footnotesize $\x$};
\clip(1,0) rectangle (7,8);

\draw [line width=2pt] (2.,7.)-- (3.,6.)-- (4.,7.)-- (5.,6.)-- (6.,5.)-- (5,4.)-- (4,4.)-- (3,5.)-- (2,4.)-- (3,3.)-- (4,3.)-- (5,3.)-- (6,2.);

\draw [line width=2pt] (2.,2.)-- (3.,2.)-- (4.,2.)-- (5.,2.)-- (6.,1.);

\draw [line width=2pt] (2.,1.)-- (3.,1.)-- (4.,1.)-- (5.,1.);

\draw [line width=2pt]  (5,1) to[out=65, in=235] (6,7);

\draw [line width=2pt, snake=snake]  (2,7) to[out=30, in=150] (6,7);
\draw [line width=2pt]  (2,1) to[out=-30, in=210] (6,1);
\draw [line width=2pt]  (2,2) to[out=30, in=150] (6,2);

\begin{scriptsize}
\draw [fill=black] (2.,1.) circle (4.5pt);
\draw [fill=black] (3.,1.) circle (4.5pt);
\draw [fill=black] (4.,1.) circle (4.5pt);
\draw [fill=black] (5.,1.) circle (4.5pt);
\draw [fill=black] (6.,1.) circle (4.5pt);
\draw [fill=black] (2.,2.) circle (4.5pt);
\draw [fill=black] (3.,2.) circle (4.5pt);
\draw [fill=black] (4.,2.) circle (4.5pt);
\draw [fill=black] (5.,2.) circle (4.5pt);
\draw [fill=black] (6.,2.) circle (4.5pt);
\draw [fill=cqcqcq] (2,3) circle (4.5pt);
\draw [fill=black] (3,3) circle (4.5pt);
\draw [fill=black] (4,3) circle (4.5pt);
\draw [fill=black] (5,3) circle (4.5pt);
\draw [fill=cqcqcq] (6,3) circle (4.5pt);
\draw [fill=black] (2,4) circle (4.5pt);
\draw [fill=cqcqcq] (3,4) circle (4.5pt);
\draw [fill=black] (4,4) circle (4.5pt);
\draw [fill=black] (5,4) circle (4.5pt);
\draw [fill=cqcqcq] (6,4) circle (4.5pt);
\draw [fill=cqcqcq] (2.,5.) circle (4.5pt);
\draw [fill=black] (3.,5.) circle (4.5pt);
\draw [fill=cqcqcq] (4.,5.) circle (4.5pt);
\draw [fill=cqcqcq] (5.,5.) circle (4.5pt);
\draw [fill=black] (6.,5.) circle (4.5pt);
\draw [fill=cqcqcq] (2.,6.) circle (4.5pt);
\draw [fill=black] (3.,6.) circle (4.5pt);
\draw [fill=cqcqcq] (4.,6.) circle (4.5pt);
\draw [fill=black] (5.,6.) circle (4.5pt);
\draw [fill=cqcqcq] (6.,6.) circle (4.5pt);
\draw [fill=black] (2.,7.) circle (4.5pt);
\draw [fill=cqcqcq] (3.,7.) circle (4.5pt);
\draw [fill=black] (4.,7.) circle (4.5pt);
\draw [fill=cqcqcq] (5.,7.) circle (4.5pt);
\draw [fill=black] (6.,7.) circle (4.5pt);
\end{scriptsize}
\end{tikzpicture}
\caption{A $23$-cycle $H$ contained in $C_5[7]$.} \label{H}
\end{minipage}%
\begin{minipage}[t]{.55\textwidth}
\centering
\begin{tikzpicture}[line cap=round,line join=round,>=triangle 45, x=.8cm,y=.9cm]
\draw[->,color=black] (1,0) -- (1,7.5);
\draw[dotted] (6.5,0) -- (6.5,7.5);

\foreach \y in {0,1,2,3,4,5,6}
\draw[shift={(1,\y+1)},color=black] (2pt,0pt) -- (-2pt,0pt) node[left] {\footnotesize $\y$};
\draw[shift={(1,8)}, color=black] node {$\mathbb{Z}_{7}$};

\draw[->,color=black] (1,0) -- (7.5, 0);
\foreach \x in {0,1,2,3,4,5}
\draw[shift={(\x+2,0)},color=black] (0pt, 2pt) -- (0pt,-2pt) node[below, below] {\footnotesize $\x$};
\clip(0,0) rectangle (8,8);

\draw [line width=2pt] (2.,7.)-- (3.,6.)-- (4.,7.)-- (5.,6.)-- (6.,5.)-- (5,4.)-- (4,4.)-- (3,5.)-- (2,4.)-- (3,3.)-- (4,3.)-- (5,3.)-- (6,2.);

\draw [line width=2pt] (2.,2.)-- (3.,2.)-- (4.,2.)-- (5.,2.)-- (6.,1.);

\draw [line width=2pt] (2.,1.)-- (3.,1.)-- (4.,1.)-- (5.,1.);

\draw [line width=2pt] (6.,7.)-- (7.,7.);
\draw [line width=2pt] (6.,1.)-- (7,1);
\draw [line width=2pt] (6.,2.)-- (7.,2.);

\draw [line width=2pt]  (5,1) to[out=65, in=235] (6,7);


\begin{scriptsize}
\draw [fill=black] (2.,1.) circle (4.5pt);
\draw [fill=black] (3.,1.) circle (4.5pt);
\draw [fill=black] (4.,1.) circle (4.5pt);
\draw [fill=black] (5.,1.) circle (4.5pt);
\draw [fill=black] (6.,1.) circle (4.5pt);
\draw [fill=black] (7.,1.) circle (4.5pt);
\draw [fill=black] (2.,2.) circle (4.5pt);
\draw [fill=black] (3.,2.) circle (4.5pt);
\draw [fill=black] (4.,2.) circle (4.5pt);
\draw [fill=black] (5.,2.) circle (4.5pt);
\draw [fill=black] (6.,2.) circle (4.5pt);
\draw [fill=black] (7.,2.) circle (4.5pt);
\draw [fill=cqcqcq] (2,3) circle (4.5pt);
\draw [fill=black] (3,3) circle (4.5pt);
\draw [fill=black] (4,3) circle (4.5pt);
\draw [fill=black] (5,3) circle (4.5pt);
\draw [fill=cqcqcq] (6,3) circle (4.5pt);
\draw [fill=cqcqcq] (7.,3.) circle (4.5pt);
\draw [fill=black] (2,4) circle (4.5pt);
\draw [fill=cqcqcq] (3,4) circle (4.5pt);
\draw [fill=black] (4,4) circle (4.5pt);
\draw [fill=black] (5,4) circle (4.5pt);
\draw [fill=cqcqcq] (6,4) circle (4.5pt);
\draw [fill=cqcqcq] (7,4) circle (4.5pt);
\draw [fill=cqcqcq] (2.,5.) circle (4.5pt);
\draw [fill=black] (3.,5.) circle (4.5pt);
\draw [fill=cqcqcq] (4.,5.) circle (4.5pt);
\draw [fill=cqcqcq] (5.,5.) circle (4.5pt);
\draw [fill=black] (6.,5.) circle (4.5pt);
\draw [fill=cqcqcq] (7.,5.) circle (4.5pt);
\draw [fill=cqcqcq] (2.,6.) circle (4.5pt);
\draw [fill=black] (3.,6.) circle (4.5pt);
\draw [fill=cqcqcq] (4.,6.) circle (4.5pt);
\draw [fill=black] (5.,6.) circle (4.5pt);
\draw [fill=cqcqcq] (6.,6.) circle (4.5pt);
\draw [fill=cqcqcq] (7.,6.) circle (4.5pt);
\draw [fill=black] (2.,7.) circle (4.5pt);
\draw [fill=cqcqcq] (3.,7.) circle (4.5pt);
\draw [fill=black] (4.,7.) circle (4.5pt);
\draw [fill=cqcqcq] (5.,7.) circle (4.5pt);
\draw [fill=black] (6.,7.) circle (4.5pt);
\draw [fill=black] (7.,7.) circle (4.5pt);
\end{scriptsize}
\end{tikzpicture}
\caption{$H^* \simeq P_{13}\ \cup\ P_5\ \cup\ P_5$} \label{H*}
\end{minipage}
\end{figure} 
 
Now, set $\Gamma = m\Z_{wm}\times \{0\}$, 
where $m\Z_{wm}=\{\lambda m\mid \lambda\in[0, w-1]\}$, and 
let $H(w)=Orb_{\Gamma}(H^*)$ be the union of all translates of $H^*$ by any element of $\Gamma$. We note that the vertices of $H(w)$ are translates of the vertices of $H$.
Figure~\ref{H(3)} shows the graph $H(3)$, with $H$ being the graph in Figure \ref{H}. In this example, 
$H(3)$ is obtained by shifting the graph $H^*$ from Figure~\ref{H*} rightwards, and it is the vertex disjoint union of three $23$-cycles.

Note that $H^*$, and hence $H(w)$, is a subgraph of $C_{wm}[n]$; also, $H^*+\gamma_1$ and $H^*+\gamma_2$, with $\gamma_1, \gamma_2\in \Gamma$, are edge-disjoint  whenever $\gamma_1\neq \gamma_2$. Therefore, $|E(H(w))| = w |E(H)|$. Finally, note that $C_m[n]^*=P_m[n]$ and $(C_m[n])(w) = C_{wm}[n]$. 
\definecolor{ccqqqq}{rgb}{0.9,0.,0.}
\definecolor{qqzzff}{rgb}{0.,0.6,1.}
\definecolor{ffffff}{rgb}{1.,1.,1.}
\definecolor{cqcqcq}{rgb}{0.75,0.75,0.75}

\noindent
\begin{figure}[h]
\centering
\begin{tikzpicture}[line cap=round,line join=round,>=triangle 45,x=0.7cm,y=1cm]
\draw[dotted] (6.5,0) -- (6.5,8);
\draw[dotted] (11.5,0) -- (11.5,8);

\draw[->,color=black] (1,0) -- (1,7.5);
\foreach \y in {0,1,2,3,4,5,6}
\draw[shift={(1,\y+1)},color=black] (2pt,0pt) -- (-2pt,0pt) node[left] {\footnotesize $\y$};
\draw[shift={(1,8)}, color=black] node {$\mathbb{Z}_{7}$};

\draw[->,color=black] (1,0) -- (16.7, 0);
\foreach \x in {0,1,2,3,4,5,6,7,8,9,10,11,12,13,14}
\draw[shift={(\x+2,0)},color=black] (0pt, 2pt) -- (0pt,-2pt) node[below, below] {\footnotesize $\x$};
\draw[shift={(17.5,0)}, color=black] node {$\mathbb{Z}_{5\cdot 3}$};

\draw [line width=2pt] (2.,7.)-- (3.,6.)-- (4.,7.)-- (5.,6.)-- (6.,5.)-- (5,4.)-- (4,4.)-- (3,5.)-- (2,4.)-- (3,3.)-- (4,3.)-- (5,3.)-- (6,2.);

\draw [line width=2pt] (7.,2.)-- (8.,2.)-- (9.,2.)-- (10.,2.)-- (11.,1.)-- (12.,1.)-- (13.,1.)-- (14.,1.)-- (15.,1.);
\draw [line width=2pt] (16.,7.);

\draw [dashed, line width=2.5pt,color=qqzzff] (2.,2.)-- (3.,2.)-- (4.,2.)-- (5.,2.)-- (6.,1.);

\draw [dashed, line width=2.5pt,color=qqzzff] (7,1)-- (8.,1.)-- (9.,1.)-- (10.,1.); 
\draw [dashed, line width=2.5pt,color=qqzzff] (11.,7.)-- (12.,7.)-- (13.,6.)-- (14.,7.)-- (15.,6.)-- (16.,5.)-- (15.,4.)-- (14.,4.)-- (13.,5.)-- (12.,4.)-- (13.0,3.0)-- (14.,3.)-- (15.,3.)-- (16.,2.);

\draw [dotted, line width=2.5pt,color=ccqqqq] (2.,1.)-- (3.,1.)-- (4.,1.)-- (5.,1.);

\draw [dotted, line width=2.5pt,color=ccqqqq] (7.,7.)-- (8.,6.)-- (9.,7.)-- (10.,6.)-- (11.,5.)-- (10.,4.0)-- (9.,4.)-- (8.,5.)-- (7.,4.)-- (8.,3.)-- (9.,3.)-- (10.,3.)-- (11.,2.)-- (12.,2.)-- (13.,2.)-- (14.,2.)-- (15.,2.)-- (16.,1.);

\draw [dotted, line width=2.5pt,  color=ccqqqq] (6.,7.)-- (7.,7.);
\draw [dashed, line width=2.5pt,  color=qqzzff] (6.,1.)-- (7,1);
\draw [line width=2pt] (6.,2.)-- (7.,2.);

\draw [dotted, color=ccqqqq, line width=2.5pt]  (5,1) to[out=55, in=235, looseness=0.6] (6,7);
\draw [dashed, color=qqzzff, line width=2.5pt]  (10,1) to[out=65, in=235] (11,7);
\draw [line width=2pt]  (15,1) to[out=65, in=235] (16,7);

\draw [dotted, color=ccqqqq, line width=2.5pt]  (2,1) to[out=270, in=270, looseness=0.2] (16,1);
\draw [dashed, color=qqzzff, line width=2.5pt]  (2,2) to[out=270, in=270, looseness=0.2] (16,2);
\draw [line width=2pt]  (2,7) to[out=90, in=90, looseness=0.2] (16,7);

\begin{scriptsize}
\draw [fill=black] (2.,1.) circle (4.5pt);
\draw [fill=black] (3.,1.) circle (4.5pt);
\draw [fill=black] (4.,1.) circle (4.5pt);
\draw [fill=black] (5.,1.) circle (4.5pt);
\draw [fill=black] (6.,1.) circle (4.5pt);
\draw [fill=black] (7.,1.) circle (4.5pt);
\draw [fill=black] (8.,1.) circle (4.5pt);
\draw [fill=black] (9.,1.) circle (4.5pt);
\draw [fill=black] (10.,1.) circle (4.5pt);
\draw [fill=black] (11.,1.) circle (4.5pt);
\draw [fill=black] (12.,1.) circle (4.5pt);
\draw [fill=black] (13.,1.) circle (4.5pt);
\draw [fill=black] (14.,1.) circle (4.5pt);
\draw [fill=black] (15.,1.) circle (4.5pt);
\draw [fill=black] (16.,1.) circle (4.5pt);
\draw [fill=black] (2.,2.) circle (4.5pt);
\draw [fill=black] (3.,2.) circle (4.5pt);
\draw [fill=black] (4.,2.) circle (4.5pt);
\draw [fill=black] (5.,2.) circle (4.5pt);
\draw [fill=black] (6.,2.) circle (4.5pt);
\draw [fill=black] (7.,2.) circle (4.5pt);
\draw [fill=black] (8.,2.) circle (4.5pt);
\draw [fill=black] (9.,2.) circle (4.5pt);
\draw [fill=black] (10.,2.) circle (4.5pt);
\draw [fill=black] (11.,2.) circle (4.5pt);
\draw [fill=black] (12.,2.) circle (4.5pt);
\draw [fill=black] (13.,2.) circle (4.5pt);
\draw [fill=black] (14.,2.) circle (4.5pt);
\draw [fill=black] (15.,2.) circle (4.5pt);
\draw [fill=black] (16.,2.) circle (4.5pt);
\draw [fill=cqcqcq] (2,3) circle (4.5pt);
\draw [fill=black] (3,3) circle (4.5pt);
\draw [fill=black] (4,3) circle (4.5pt);
\draw [fill=black] (5,3) circle (4.5pt);
\draw [fill=cqcqcq] (6,3) circle (4.5pt);
\draw [fill=cqcqcq] (7.,3.011276862583572) circle (4.5pt);
\draw [fill=black] (8.,3.011276862583572) circle (4.5pt);
\draw [fill=black] (9.,3.011276862583572) circle (4.5pt);
\draw [fill=black] (10.,3.011276862583572) circle (4.5pt);
\draw [fill=cqcqcq] (11.,3.011276862583572) circle (4.5pt);
\draw [fill=cqcqcq] (12.,3.011276862583572) circle (4.5pt);
\draw [fill=black] (13.,3.011276862583572) circle (4.5pt);
\draw [fill=black] (14.,3.011276862583572) circle (4.5pt);
\draw [fill=black] (15.,3.011276862583572) circle (4.5pt);
\draw [fill=cqcqcq] (16.,3.011276862583572) circle (4.5pt);
\draw [fill=black] (2,4) circle (4.5pt);
\draw [fill=cqcqcq] (3,4) circle (4.5pt);
\draw [fill=black] (4,4) circle (4.5pt);
\draw [fill=black] (5,4) circle (4.5pt);
\draw [fill=cqcqcq] (6,4) circle (4.5pt);
\draw [fill=black] (7,4) circle (4.5pt);
\draw [fill=cqcqcq] (8.022553725167144,4.011276862583572) circle (4.5pt);
\draw [fill=black] (9.022553725167144,4.011276862583572) circle (4.5pt);
\draw [fill=black] (10.022553725167144,4.011276862583572) circle (4.5pt);
\draw [fill=cqcqcq] (11.022553725167144,4.011276862583572) circle (4.5pt);
\draw [fill=black] (12.022553725167144,4.011276862583572) circle (4.5pt);
\draw [fill=cqcqcq] (13.022553725167144,4.011276862583572) circle (4.5pt);
\draw [fill=black] (14.022553725167144,4.011276862583572) circle (4.5pt);
\draw [fill=black] (15.022553725167144,4.011276862583572) circle (4.5pt);
\draw [fill=cqcqcq] (16.022553725167143,4.011276862583572) circle (4.5pt);
\draw [fill=cqcqcq] (2.,5.) circle (4.5pt);
\draw [fill=black] (3.,5.) circle (4.5pt);
\draw [fill=cqcqcq] (4.,5.) circle (4.5pt);
\draw [fill=cqcqcq] (5.,5.) circle (4.5pt);
\draw [fill=black] (6.,5.) circle (4.5pt);
\draw [fill=cqcqcq] (7.,5.) circle (4.5pt);
\draw [fill=black] (8.,5.) circle (4.5pt);
\draw [fill=cqcqcq] (9.,5.) circle (4.5pt);
\draw [fill=cqcqcq] (10.,5.) circle (4.5pt);
\draw [fill=black] (11.,5.) circle (4.5pt);
\draw [fill=cqcqcq] (12.,5.) circle (4.5pt);
\draw [fill=black] (13.,5.) circle (4.5pt);
\draw [fill=cqcqcq] (14.,5.) circle (4.5pt);
\draw [fill=cqcqcq] (15.,5.) circle (4.5pt);
\draw [fill=black] (16.,5.) circle (4.5pt);
\draw [fill=cqcqcq] (2.,6.) circle (4.5pt);
\draw [fill=black] (3.,6.) circle (4.5pt);
\draw [fill=cqcqcq] (4.,6.) circle (4.5pt);
\draw [fill=black] (5.,6.) circle (4.5pt);
\draw [fill=cqcqcq] (6.,6.) circle (4.5pt);
\draw [fill=cqcqcq] (7.,6.) circle (4.5pt);
\draw [fill=black] (8.,6.) circle (4.5pt);
\draw [fill=cqcqcq] (9.,6.) circle (4.5pt);
\draw [fill=black] (10.,6.) circle (4.5pt);
\draw [fill=cqcqcq] (11.,6.) circle (4.5pt);
\draw [fill=cqcqcq] (12.,6.) circle (4.5pt);
\draw [fill=black] (13.,6.) circle (4.5pt);
\draw [fill=cqcqcq] (14.,6.) circle (4.5pt);
\draw [fill=black] (15.,6.) circle (4.5pt);
\draw [fill=cqcqcq] (16.,6.) circle (4.5pt);
\draw [fill=black] (2.,7.) circle (4.5pt);
\draw [fill=cqcqcq] (3.,7.) circle (4.5pt);
\draw [fill=black] (4.,7.) circle (4.5pt);
\draw [fill=cqcqcq] (5.,7.) circle (4.5pt);
\draw [fill=black] (6.,7.) circle (4.5pt);
\draw [fill=black] (7.,7.) circle (4.5pt);
\draw [fill=cqcqcq] (8.,7.) circle (4.5pt);
\draw [fill=black] (9.,7.) circle (4.5pt);
\draw [fill=cqcqcq] (10.,7.) circle (4.5pt);
\draw [fill=black] (11.,7.) circle (4.5pt);
\draw [fill=black] (12.,7.) circle (4.5pt);
\draw [fill=cqcqcq] (13.,7.) circle (4.5pt);
\draw [fill=black] (14.,7.) circle (4.5pt);
\draw [fill=cqcqcq] (15.,7.) circle (4.5pt);
\draw [fill=black] (16.,7.) circle (4.5pt);
\end{scriptsize}
\draw[dotted] (6.5,0) -- (6.5,7.5);
\draw[dotted] (11.5,0) -- (11.5,7.5);
\end{tikzpicture}
\caption{$H(3) \simeq C_{23}\ \cup\ C_{23}\ \cup\ C_{23}$} \label{H(3)}
\end{figure}
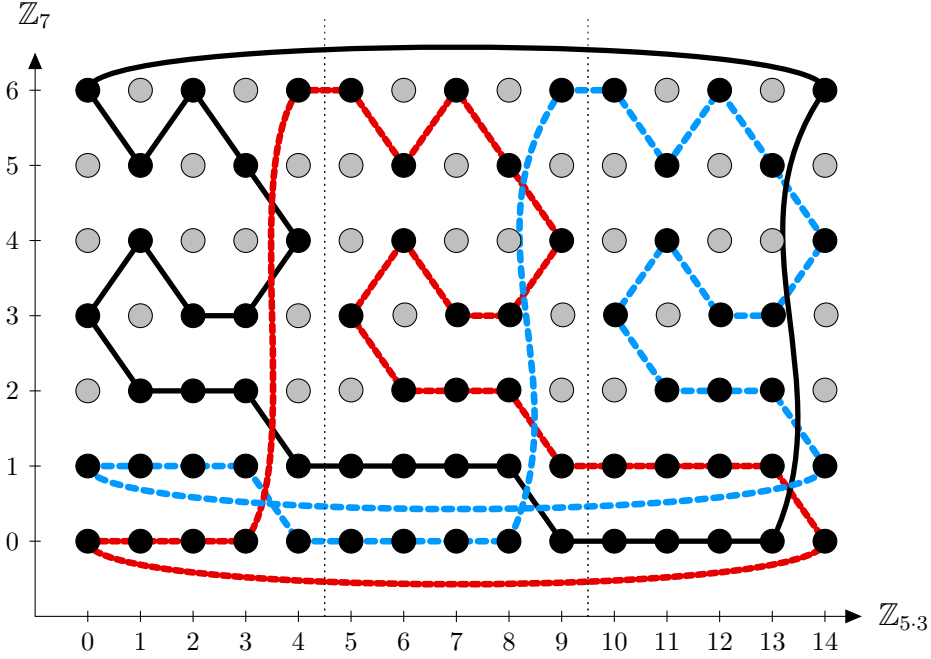

\begin{lemma}\label{properties:H(w)}
  Let $H$ be a subgraph of $C_m[n]$, and for 
  each $i\in [1,t]$ let $H_i$ be a subgraph of $H$. If $w\geq 2$, then the following properties hold:
  \begin{enumerate}
    \item if $H$ is $d$-regular, then $H(w)$ is $d$-regular; \label{properties:H(w):1} 
    \item if the $H_i$s are edge-disjoint (resp., vertex-disjoint), 
    then the $H_i(w)$s are edge-disjoint (resp., vertex-disjoint);\label{properties:H(w):2}
    \item if $H= \bigcup_{i=1}^t H_i$, then $H(w)= \bigcup_{i=1}^t H_i(w)$. \label{properties:H(w):3}
  \end{enumerate}
	Thus, if the $H_i$ are a $d$-regular factorization of $H$, then the $H_i(w)$ are a $d$-regular factorization of $H(w)$.
\end{lemma}
\begin{proof} Let $H$ be a subgraph of $C_m[n]$, and let $H_1,\ldots, H_t$ be subgraphs of $H$.

We start by proving property 1. First, let
$(x,y)\in V(H^*)$ and note that if $x\in[1,m-1]$, then 
$deg_{H^*}(x,y) = deg_{H}(x,y) = d$; if $x\in\{0,m\}$, it is easy to see that $deg_{H^*}(0,y) + deg_{H^*}(m,y) = d$.
Now, let $(x'+\lambda m, y')\in V(H(w))$ with $x'\in[0,m-1]$ and 
$\lambda\in[0,w-1]$. Clearly,
\[
deg_{H(w)}(x'+\lambda m, y') = 
\sum_{\gamma\in\Gamma} deg_{H^*+\gamma}(x' + \lambda m, y').
\] 
Also, we have that 
$(x'+\lambda m, y')\in V(H^*+\gamma$), 
with $\gamma\in\Gamma$, if and only if $x'=0$ and 
$\gamma \in \{(\lambda m- m,0),(\lambda m,0)\}$, or 
$x'\in[1,m-1]$ and $\gamma = (\lambda m, 0)$. Therefore,
if $x'\in[1,m-1]$, then $deg_{H(w)}(x'+\lambda m, y') = 
deg_{H^*+(\lambda m,0)}(x'+\lambda m, y') = 
deg_{H^*}(x', y') = d.$
If $x'=0$, then 
\begin{align*}
    deg_{H(w)}(\lambda m, y') & = 
    deg_{H^*+(\lambda m-m,0)}(\lambda m, y') + 
    deg_{H^*+(\lambda m,0)}(\lambda m, y') \\
  & = deg_{H^*}(m, y') + 
      deg_{H^*}(0, y') = d.
\end{align*}
It follows that $H(w)$ is $d$-regular.

To prove property 2,
we start showing that if the $H_i$s are pairwise edge-disjoint, then $H^*_i$ and $H^*_j + \gamma$ are edge-disjoint whenever $i\neq j$ and $\gamma\in \Gamma$. 
Assume for a contradiction that $E(H^*_i)\ \cap\ E(H^*_j)=\emptyset$ and $E(H^*_i)\ \cap\ E(H^*_j+\gamma)\neq \emptyset$ for some $i\neq j$ and 
$\gamma=(\lambda m,0)$, with $\lambda\in[0,w-1]$. Then, there exist
$(x_1,y_1)(x_2,y_2)\in E(H^*_i)$ and $(x'_1,y'_1)(x'_2,y'_2)\in E(H^*_j)$ such that 
\[
(x_1,y_1)(x_2,y_2)= (x'_1+\lambda m,y'_1)(x'_2+\lambda m,y'_2).
\] 
Without loss of generality, we can assume that $x_u = x'_u+\lambda m$, that is, 
$x_u - x'_u = \lambda m$ for $u\in\{1,2\}$. 
Since $x_1,x_1', x_2,x_2'\in[0,m]$ and recalling that $H^*_i$ and $H^*_j$ are edge-disjoint, we necessarily have that $(x_1, x_1')=(x_2, x_2')=(m,0)$, but this is a contradiction since 
$x_1- x_2, x'_1- x'_2\in\{\pm1\}$.
It follows that $E(H^*_i+\gamma_1)\ \cap\ E(H^*_j+\gamma_2) = \emptyset$ for every $i\neq j$ and for every 
$\gamma_1, \gamma_2\in \Gamma$. Therefore, 
$H_i(w)= Orb_{\Gamma}(H_i^*)$ and $H_j(w)= Orb_{\Gamma}(H_j^*)$ are edge-disjoint for every $i\neq j$.

Similarly, one can prove that if the $H_i$s are pairwise vertex-disjoint, then 
$V(H^*_i+\gamma_1)\ \cap\ V(H^*_j+\gamma_2) = \emptyset$ for every $i\neq j$ and for every 
$\gamma_1, \gamma_2\in \Gamma$, hence the $H_i(w)$s are pairwise vertex-disjoint.

We finally prove property 3. If $H = \bigcup_{i=1}^t H_i$, it is easy to see that  $H^* = \bigcup_{i=1}^t H_i^*$, hence 
$H(w) = Orb_{\Gamma}(H^*)
     = \bigcup_{i=1}^t Orb_{\Gamma}(H_i^*)
     = \bigcup_{i=1}^t H_i(w),$
and this completes the proof.
\end{proof}

We now explore the structure of expanded cycles.
Let \linebreak $C=\big((x_0,y_0), (x_1, y_1), \ldots, (x_{\ell-1}, y_{\ell-1})\big)$ be 
an $\ell$-cycle in $C_m[n]$, where the subscripts of $x_i$ and $y_i$ are to be considered modulo $\ell$, 
and set $(x_\ell, y_\ell) = (x_0,y_0)$. 
We define a trail $T_C$ in the graph $C_{mw}[n]$ as follows:
\begin{align*}
  T_C= \left\langle
     (X_0,y_0), (X_1, y_1), \ldots, (X_{\ell-1}, y_{\ell-1}), (X_{\ell}, y_{\ell}) \right\rangle
\end{align*}
with $X_j = x_j + \lambda_j m$, where $\lambda_0 = 0$ and for every $j\in[1,\ell]$ we have that
\begin{align*}
 & \lambda_j  = \lambda_{j-1} + 
  \begin{cases}
    1 & \text{if $(x_{j-1}, x_j)=(m-1,0)$}, \\
   -1 & \text{if $(x_{j-1}, x_j)=(0,m-1)$}, \\ 
    0 & \text{otherwise}.  
  \end{cases}
\end{align*}
Furthermore,
we define the parameter $\epsilon(C)$ as follows:
\small
\[
\epsilon(C) = \bigg|
\left|\{i\in \Z_\ell\mid (x_i, x_{i+1})=(m-1,0)\}\big|-
\right|\{i\in \Z_\ell\mid (x_i, x_{i+1})=(0, m-1)\}\big|
\bigg|
.
\]
\normalsize
For example, for the cycle $H$ given in Figure~\ref{H}, $\epsilon(H)=3$ and each of the cycles in 
Figure~\ref{H(3)} can be seen as a trail $T_H$ arising from $H$. Note that
$\lambda_\ell = \epsilon(C)$, therefore $X_{\ell} = X_0 + \epsilon(C) m$.

The trail $T_C$ and $\epsilon(C)$ are used in the following lemma to determine the length of the cycles
in $C(w)$.
\begin{lemma}\label{lengthC(w)}
  If $C$ is an $\ell$-cycle of $C_m[n]$, then for every $w \geq 1$, $C(w)$ consists of $u$ vertex-disjoint cycles of length $k$, where $u=\gcd(\epsilon(C), w)$ and $k=\ell w/u$.
\end{lemma}
\begin{proof} 
Let $T=T_C$ be a trail arising from $C$ as described above.
It is not difficult to check that 
$(X_j,y_j)\neq (X_{j'},y_{j'})$ whenever $0\leq j\neq j'\leq \ell-1$. Therefore, the trail 
$T$ is either a cycle or a path according to whether $X_\ell=X_0$, or not. 

To prove that $C(w)=Orb_\Gamma(T)$, we start by showing that if 
$e\in E(C^*)$, then $e\in E(T+\gamma)$, for some $\gamma\in\Gamma$. 
Let $e\in E(C^*)$ and consider the following cases.
\begin{enumerate}
  \item If $e\in E(C)$, then $e=(x_{j-1}, y_{j-1})(x_{j}, y_{j})$ for some $j\in [1,\ell]$, and
  $(m-1,0)\neq (x_{j-1},x_{j})\neq(0,m-1)$. It follows that $\lambda_{j}= \lambda_{j-1}$,
  hence $e+(\lambda_{j}m,0) = (X_{j-1}, y_{j-1})(X_{j}, y_{j})\in T$, that is,
  $e\in E(T-(\lambda_{j}m,0))$.
  
  \item If $e\not\in E(C)$, then one of the following two cases holds.
  
  \textbf{Case 1:} $e=(x_{j-1}, y_{j-1})(x_j+m, y_{j})$ with $(x_{j-1}, x_{j}) = (m-1,0)$, 
    for some $j\in [1,\ell]$. In this case, $\lambda_j = \lambda_{j-1}+1$. Hence, 
    $e+(\lambda_{j-1}m,0) = (X_{j-1}, y_{j-1})(X_{j}, y_{j})\in T$, that is,
    $e\in E(T-(\lambda_{j-1}m,0))$.
   
  \textbf{Case 2:} $e=(x_{j-1}+m, y_{j-1})(x_j, y_{j})$ with $(x_{j-1}, x_{j}) = (0, m-1)$,
    for some $j\in [1,\ell]$. Therefore, $\lambda_j = \lambda_{j-1}+1$, and
    $e+(\lambda_{j}m,0) = (X_{j-1}, y_{j-1})(X_{j}, y_{j})\in T$, that is, 
    $e\in E(T-(\lambda_{j}m,0))$. 
\end{enumerate}
It then follows that $Orb_{\Gamma}(e)\subseteq Orb_{\Gamma}(T)$ for every $e\in C^*$, hence
\[ C(w) = Orb_{\Gamma}(C^*) = \bigcup_{e\in E(C^*)} Orb_\Gamma(e) \subseteq Orb_{\Gamma}(T).
\]
Note that
$|E(C(w))|\leq |E(Orb_{\Gamma}(T))| \leq w|E(T)| = w\ell = |E(C(w))|$, therefore 
$C(w)=Orb_\Gamma(T)$.

By Lemma \ref{properties:H(w)}.(\ref{properties:H(w):1}),
we know that $C(w)$ is $2$-regular. Hence, it is left to show that 
$C(w)$ consists only of cycles of length $\ell w/u$, where 
$u=\gcd(\epsilon(C), w)$.
Considering that $\lambda_\ell = \epsilon(C)$, it follows that $X_{\ell} = X_0 + \epsilon(C) m$. Hence every cycle of $Orb_\Gamma(T)$ is made of $p$ copies of $T$ where $p$ is the first positive integer such that 
$p\epsilon(C) \equiv 0 \pmod{w}$, that is, $p=w/u$.
Thus, each cycle has length $\ell p = \ell w /u$.
\end{proof}

\begin{lemma} \label{espilonCodd}
  Let $C$ be an $\ell$-cycle in $C_m[n]$. If $\ell$ is odd, then $\epsilon(C)$ is odd. 
\end{lemma}
\begin{proof}
 We note that when $m$ is even, then $C_m[n]$ is bipartite. Since, by assumption,
$C_m[n]$ contains a cycle $C$ of odd length $\ell$, then $m$ must be odd.
   Now, set 
  $\mathcal{E}=\big\{(x,y)(x',y')\in E(C)\mid x,x'\in\{0,m-1\}\big\}$.
  It is clear that $\epsilon(C)\equiv |\mathcal{E}| \pmod{2}$ and the components of $C\setminus \mathcal{E}$ are paths, which we denote by $Q_1, Q_2, \ldots, Q_t$, of lengths $\ell_1, \ldots, \ell_t$, respectively. Considering that $Q_i$ is a subgraph of $P_{m-1}[n]$, and its ends lie in $\{0,m-1\}\times \Z_n$, 
    it is easy to see that $\ell_i \equiv m-1 \equiv 0$ (mod 2) for every $i \in [1,t]$.  
Since $\ell = |\mathcal{E}|+\ell_1 + \cdots+ \ell_t$ is odd, it follows that $|\mathcal{E}|$ is odd, and this proves the assertion.
\end{proof}

\begin{theorem} \label{ASSW extension}
Let $\ell_1, \ell_2,m$ and $n$ be positive integers with $\ell_1, \ell_2, m\geq 3$ odd, and let $S$ be a subset of $\Z_n$. If HWP$(C_m[S]; \ell_1, \ell_2;$ $\alpha, \beta)$ has a solution, 
then there is a solution to  HWP$(C_{2^tm}[S]; 2^t\ell_1, 2^t\ell_2; \alpha, \beta)$.
\end{theorem}
\begin{proof} 
  Let $\mathcal{F}=\{F_1, F_2, \ldots, F_s\}$ be a solution to 
  HWP$(C_m[S]; \ell_1, \ell_2;$ $\alpha, \beta)$, where $s=|S|$, and denote by 
  $C_{i,1}, C_{i,2}, \ldots, C_{i,t_i}$ the components of $F_i$, for $i\in[1,s]$.
  By Lemma \ref{properties:H(w)},
	$\mathcal{F'}=\{F_1(2^t), F_2(2^t), \ldots, F_s(2^t)\}$ is a $2$-factorization of $(C_m[S])(2^t)=C_{2^tm}[S]$.
  
  By Lemma \ref{espilonCodd}, $\epsilon(C_{ij})$ is odd; hence, by Lemma 
  \ref{lengthC(w)}, we have that $C_{ij}(2^t)$ is a cycle of length $2^t|C_{ij}|$. Finally, by Lemma \ref{properties:H(w)}.(\ref{properties:H(w):2}), we have that the $C_{ij}(2^t)$s are the components of $F_i(2)$, for $i\in[1,s]$.
  It follows that $\mathcal{F}'$ has as many $C_{2^t\ell_u}$-factors as the $C_{\ell_u}$-factors of 
  $\mathcal{F}$. Therefore, $\mathcal{F}'$ is a solution to HWP$(C_{2^tm}[S]; 2^t\ell_1, 2^t\ell_2; \alpha, \beta)$.
\end{proof}

\section{Constructing 2-factorizations of blown-up cycles} 
\label{Section blown up cycle}

Let $n > m \geq 3$ be odd integers with $m \nmid n$, and let $g=\gcd(m,n)$.  We first prove a result that is a consequence of Theorem~\ref{ASSW extension} and the following theorem, which is a special case of Theorem 1.4 of~\cite{GOP}, taking $t=1,2$.

\begin{theorem}[\cite{GOP}] \label{GOP theorem}
Let $1 \leq m' < n' \leq N$ be odd integers such that $m'$ and $n'$ are divisors of $N$.  Then $\hwp(C_g[N];gm',gn';\alpha,\beta)$ has a solution whenever $g \geq 3$, $\alpha+\beta=N$ and $\alpha,\beta \neq 1$.
\end{theorem}

We now apply Theorem~\ref{ASSW extension} to obtain the following result.  

\begin{theorem} \label{C_{2g}[N]}
Let $3 \leq m < n$ be odd integers, let $g=\gcd(m,n)$, $m'=m/g$ and $n'=n/g$.  If $g>1$, $m' \mid N$ and $n' \mid N$, then $\hwp(C_{2g}[N];2m,2n;\alpha,\beta)$ has a solution whenever $\alpha+\beta=N$ and $\alpha,\beta \neq 1$.
\end{theorem}

In the remainder of this section, we will solve $\hwp(C_{2m/g}[n];2m,2n;\alpha,\beta)$, except possibly when $\beta \in \{1,3\}$, or $\alpha=1$ and $g>1$.  
We first recall the following result from~\cite{BDT2}.  
\begin{theorem}[\cite{BDT2}] \label{C_{m/g} theorem}
Let $m$ and $n$ be odd integers with $n>m \geq 3$, let $g \neq m$ be a common divisor of $m$ and $n$, and let $\alpha$ and $\beta$ be nonnegative integers.  There is a solution to $\hwp(C_{m/g}[n];m,n;\alpha,\beta)$ whenever $\alpha+\beta=n$, except possibly if $\beta \in \{1,3\}$ or if one of the following conditions hold:
\begin{enumerate}
\item $g=1$ and either $(\alpha,\gcd(m!,n))=(2,1)$ or $(\alpha,m)=(4,3)$;
\item $g>1$ and $\alpha=1$.
\end{enumerate}
\end{theorem}

Combining Theorem~\ref{C_{m/g} theorem} with Theorem~\ref{ASSW extension} gives the following:
\begin{theorem}\label{C_{m/g} extended}
Let $m$ and $n$ be odd integers with $n>m \geq 3$, let $g \neq m$ be a common divisor of $m$ and $n$, and let $\alpha$ and $\beta$ be nonnegative integers.  There is a solution to $\hwp(C_{2m/g}[n];2m,2n;\alpha,\beta)$ whenever $\alpha+\beta=n$, except possibly if $\beta \in \{1,3\}$ or if one of the following conditions hold:
\begin{enumerate}
\item $g=1$ and either $(\alpha,\gcd(m!,n))=(2,1)$ or $(\alpha,m)=(4,3)$;
\item $g>1$ and $\alpha=1$.
\end{enumerate}
\end{theorem}

In the case $g=1$, however, we can improve this result, removing many of the exceptions and allowing the possibility that $m$ is even in some cases.

We first describe the formation of $C_{2m}$-factors of $C_{2m}[n]$.  The following theorem is proved in~\cite[Theorem 2.11]{BDT1}.
\begin{theorem} \label{Matrix theorem}
Let $T$ be a subset of $\mathbb{Z}_n$ and $\ell \geq 3$.  If there exists a $|T| \times \ell$ matrix $A=[a_{ij}]$ with entries from $T$ such that:
\begin{enumerate}
\item Each row of $A$ sums to 0, and
\item Each column of $A$ is a permutation of $T$,
\end{enumerate}
then there exists a $C_{\ell}$-factorization of $C_{\ell}[T]$.  Moreover, if we also have that:
\begin{enumerate}
\item[3.] $T$ is closed under taking negatives,
\end{enumerate}
then there is a $C_m$-factorization of $C_m[T]$ for any $m \geq \ell$ with $m \equiv \ell \pmod{2}$.
\end{theorem}

\begin{corollary} \label{C_{2m}-factors}
Let $d_1, d_2, \ldots, d_k$ be distinct elements of $\{1, 2, \ldots, \lfloor \frac{n}{2}\rfloor\}\subset~\Z_n$.  There is a $C_{2m}$-factorization of $C_{2m}[S]$, where $S = \{0, \pm d_1, \pm d_2, \ldots, \pm d_k\}$ or 
$\{\pm d_1, \pm d_2, \ldots, \pm d_k\}$.  
\end{corollary}

\begin{proof}
The result follows from the fact that the matrices 
\[
A = \left[ \begin{array}{rrrr} 
    0 & 0 & 0 & 0\\ 
    d_1 & -d_1 & d_1 & -d_1  \\ 
    -d_1 & d_1 & -d_1 & d_1  \\ 
    d_2 & -d_2 & d_2 & -d_2  \\ 
    -d_2 & d_2 & -d_2 & d_2  \\ 
    \vdots & \vdots  & \vdots & \vdots \\ 
    d_k & -d_k & d_k & -d_k  \\ 
    -d_k & d_k & -d_k & d_k 
    \end{array} \right]
\mbox{ and }
B = \left[\begin{array}{rrrr} 
    d_1 & -d_1 & d_1 & -d_1\\ 
    -d_1 & d_1 & -d_1 & d_1\\ 
    d_2 & -d_2 & d_2 & -d_2\\ 
    -d_2 & d_2 & -d_2 & d_2\\ 
    \vdots & \vdots & \vdots & \vdots\\ 
    d_k & -d_k & d_k & -d_k\\ 
    -d_k & d_k & -d_k & d_k
    \end{array} \right]
\]
satisfy conditions 1--3 of Theorem~\ref{Matrix theorem}.
\end{proof}
In particular, Corollary~\ref{C_{2m}-factors} guarantees the existence of $C_{2m}$-factorizations of $C_{2m}[\pm\{0, 1, 2, \ldots, \pm k\}]$ and $C_{2m}[\{\pm 1, \pm 2, \ldots, \pm k\}]$ for any $k \in \{1, 2, \ldots, \lfloor \frac{n}{2}\rfloor\}$.

Next, we describe the formation of $C_{2n}$-factors of $C_{2m}[n]$.  First, recall the following result regarding $C_n$-factors in $C_m[n]$.

\begin{theorem}[\cite{ASSW, BDT1}] \label{C_n factors}
Let $m$ and $n$ be odd integers with $3 \leq m \leq n$.  There is a $C_n$-factorization of $C_m[n,S]$, where either $S=\pm 	\{0,1,  2\}$ or $S=\pm \{w, \ldots, \frac{n-1}{2}\}$ for some $w \in \{0, \ldots, \frac{n-1}{2}\}$.
\end{theorem}

As a consequence of Theorems~\ref{C_n factors} and \ref{ASSW extension}, we now have the following.

\begin{lemma} \label{012} \label{C_{2n}-factors}
Let $m$ and $n$ be odd integers with $3 \leq m \leq n$.  There is a $C_{2n}$-factorization of $C_{2m}[n,S]$, where either $S=\pm \{0, 1, 2\}$ or $S=\pm \{w, \ldots, \frac{n-1}{2}\}$ for some $w \in \{0, \ldots, \frac{n-1}{2}\}$.
\end{lemma}

We can now provide solutions to $\hwp(C_{2m}[n];2m,2n;\alpha,\beta)$.
The following theorems improve on the result of Theorem~\ref{C_{m/g} extended} in the case that $g=1$.  In particular, note that Theorem~\ref{C_{2m}[n] odd alpha} allows for the possibility that $m$ is even.
\begin{theorem} \label{C_{2m}[n] odd alpha}
Let $m,n \geq 2$ be positive integers with $n$ odd, and let $0<\alpha \leq n$ be odd.  Then there is a solution to $\hwp(C_{2m}[n];2m,2n; \alpha,n-\alpha)$.  
\end{theorem}
\begin{proof}
First decompose $C_{2m}[n]=G_1 \oplus G_2$, where $G_1=C_{2m}[\pm \{0, 1, \ldots, \frac{\alpha-1}{2}\}]$ and $G_2=C_{2m}[\pm \{\frac{\alpha+1}{2}, \ldots, \frac{n-1}{2}\}]$.  (Note that if $\alpha=n$, then $G_2$ has no edges.)  Now, $G_1$ admits a $C_{2m}$-factorization by Corollary~\ref{C_{2m}-factors}, and $G_2$ admits a $C_{2n}$-factorization by Lemma~\ref{C_{2n}-factors}.
\end{proof}

\begin{theorem} \label{C_{2m}[n] even alpha}
Let $3 \leq m \leq n$ be odd integers and let $0 \leq \alpha \leq n-5$ be even.  Then there is a solution to $\hwp(C_{2m}[n];2m,2n;\alpha,n-\alpha)$.
\end{theorem}
\begin{proof}
First decompose $C_{2m}[n] = G_1 \oplus G_2 \oplus G_3$, where $G_1=C_{2m}[0, \pm 1, \pm 2]$, $G_2=C_{2m}[\pm 3, \pm 4, \ldots, \pm \frac{\alpha+4}{2}]$ and $G_3=C_{2m}[\pm \frac{\alpha+6}{2}, \ldots, \pm \frac{n-1}{2}]$.  We have that $G_1$ factors into five $C_{2n}$-factors and $G_3$ factors into $n-\alpha-5$ $C_{2n}$-factors by Lemma~\ref{C_{2n}-factors}, giving $n-\alpha$ $C_{2n}$-factors in total.  Finally, $G_2$ factors into $\alpha$ $C_{2m}$-factors by Corollary~\ref{C_{2m}-factors}.
\end{proof}

In the case that $m$ and $n$ are both odd, Theorems~\ref{C_{m/g} extended} and \ref{C_{2m}[n] even alpha} give the following result regarding factorizations of $C_{2m/g}[n]$.
\begin{theorem} \label{decomposing C_{2m/g}[n]}
Let $m$ and $n$ be odd integers with $n>m \geq 3$, let $g \neq m$ be a common divisor 
 of $m$ and $n$, and let $\alpha$ and $\beta$ be nonnegative integers.  There is a solution to $\hwp(C_{2m/g}[n];2m,2n;\alpha,\beta)$ whenever $\alpha+\beta=n$, except possibly if at least  one of the following conditions hold:
\begin{enumerate}
\item $\beta \in \{1,3\}$;
\item $g>1$ and $\alpha=1$.
\end{enumerate}
\end{theorem}

\section{Main Results}
\label{Section Main}

\begin{lemma} \label{Main Lemma}
Let $m$ and $n$ be odd positive integers with $n>m \geq 3$, let $G$ be a regular graph of degree $2(\alpha+\beta)>2n$ with $\alpha, \beta >0$, and let $g \neq m$ be a divisor of $\gcd(m,n)$.  If $G$ has a $C_{2m/g}[n]$-factorization, then there is a solution to $\hwp(G;2m,2n;\alpha,\beta)$, except possibly if at least one of the following holds:
\begin{enumerate}
\item $\beta \in \{1,3\}$;
\item $g>1$ and $\alpha=1$.
\end{enumerate}
\end{lemma}

\begin{proof}
Let $m'=m/g$, and let $\mathcal{G} = \{G_1, G_2, \ldots, G_r\}$ be a set of $C_{2m'}[n]$-factors which decompose $G$.  Note that the existence of such a factorization implies that $m'>1$ (so that $g \neq m$).  Also, $r=\frac{\alpha+\beta}{n} \geq 2$.  

For each $i \in \{1, \ldots, r\}$, we will define appropriate values $\alpha_i, \beta_i$ such that $\alpha_1 + \cdots + \alpha_r=\alpha$, $\beta_1 + \cdots + \beta_r = \beta$ and $\alpha_i + \beta_i=n$.  We then factor each $G_i$ into $\alpha_i$ $C_{2m}$-factors and $\beta_i$ $C_{2n}$-factors by using Theorem~\ref{decomposing C_{2m/g}[n]} to solve $\mathrm{HWP}(C_{2m'}[n]; 2m,2n; \alpha_i, \beta_i)$.
 
Write $\alpha=xn+y$, where $0 \leq x < r$ and $0 \leq y < n$.  First, if $y \notin \{1, n-3, n-1\}$ or if $y=1$ and $g=1$, define
\[
\alpha_i = \left\{ \begin{array}{ll}
n, & \mbox{if } 1 \leq i \leq x \\
y, & \mbox{if } i=x+1 \\
0, & \mbox{if } x+2 \leq i \leq r.
\end{array}
\right.
\]
Otherwise, if $y=1$ and $g>1$, noting that by Exception 2 we have $x \geq 1$, define
\[
\alpha_i = \left\{ \begin{array}{ll}
n, & \mbox{if } 1 \leq i \leq x-1 \\
n-2, & \mbox{if } i=x \\
3, & \mbox{if } i=x+1 \\
0, & \mbox{if } x+2 \leq i \leq r.
\end{array}
\right.
\]
Finally, supposing $y \in \{n-3, n-1\}$, then by Exception 1 we have $x \leq r-2$.  If $g=1$, let $z=1$; otherwise (noting that $n>9$) let $z=3$.  Define 
\[
\alpha_i = \left\{ \begin{array}{ll}
n, & \mbox{if } 1 \leq i \leq x \\
y-z, & \mbox{if } i=x+1 \\
z, & \mbox{if } i=x+2 \\
0, & \mbox{if } x+3 \leq i \leq r,
\end{array}
\right.
\]

In any case, we define $\beta_i = n-\alpha_i$ for all $i \in \{1, \ldots, r\}$.  It is easy to verify in each case that Theorem~\ref{decomposing C_{2m/g}[n]} guarantees the existence of a solution of $\hwp(C_{2m/g}[n];2m,2n;\alpha_i,\beta_i)$ for each $i \in \{1, \ldots, r\}$, proving the result.
\end{proof}

\begin{theorem} \label{multipartite}
Let $t$ and $w$ be positive integers with $w$ even and $t \geq 3$. Let $m$ and $n$ be odd divisors of $w$ with $n>m \geq 3$ and $m \nmid n$, and suppose that $\alpha$ and $\beta$ are positive integers such that $2(\alpha+\beta)=(t-1)w$.  There is a solution to $\hwp(K_t[w];2m,2n;\alpha,\beta)$, except possibly if at least one of the following holds:
\begin{enumerate}
\item $\beta =1$; 
\item $\beta=3$ and $\gcd(m,n)=1$;  
\item $\alpha=1$ and $mn \nmid w$.  
\end{enumerate}
\end{theorem}

\begin{proof}
Let $g=\gcd(m,n)$, $m=m'g$ and $n=n'g$ for some odd positive integers $m', n'$.
Since $m$ and $n$ are divisors of $w$, $mn|wg$; hence $w=2^xm'n'gs=2^xm'ns$ for some $x \geq 1$ and odd $s>0$. 

We first suppose $\beta \neq 3$.  By Theorem~\ref{liu},
there exists a $C_{2m/g}$-factorization of $K_t[w/n]$, and hence $K_t[w]$ admits a $C_{2m/g}[n]$-factorization.
Note that $w>2n$ since $m'>1$, and hence $2(\alpha+\beta)=(t-1)w>2n$. The result now follows by Lemma~\ref{Main Lemma}, except if 
$\alpha=1$ and  $g>1$.  Note that if $mn\nmid w$, then $g>1$. Therefore, it is left to show that there is a solution to
$\hwp(K_t[w];2m,2n;\alpha,\beta)$ when $\alpha=1$ and $mn\mid w$. In this case,
we use Theorem~\ref{liu} to find a $C_{2m}$-factorization of $K_t[w/n]$, blow up by $n$ to obtain a $C_{2m}[n]$-factorization of $K_{t}[w]$ and again apply Lemma~\ref{Main Lemma}.

If $\beta=3$, then we note that $K_t[w/(m'n')]$ admits a $C_{2g}$-factorization by Theorem~\ref{liu}, and hence $K_t[w]$ admits a $C_{2g}[m'n']$-factorization.  We fill one $C_{2g}[m'n']$-factor with a solution to $\hwp(C_{2g}[m'n'];2m,2n;m'n'-3,3)$ and the rest with a solution to $\hwp(C_{2g}[m'n'];2m,2n;m'n',0)$, both of which exist by Theorem~\ref{C_{2g}[N]} when $g>1$.
\end{proof}

\begin{theorem} \label{main result}
Let $v \equiv 2$ (mod 4), let $n > m \geq 3$ be odd integers and let $\alpha, \beta$ be odd positive integers.  There is a solution to $\hwp(v;2m,2n;\alpha,\beta)$ if and only if $m \mid v$, $n \mid v$ and $\alpha+\beta = \frac{v-2}{2}$, except possibly if 
$m\nmid n$ and at least one of the following holds:
\begin{enumerate}
\item $\beta =1$; \label{beta1}
\item $\beta=3$ and $\gcd(m,n)=1$; \label{beta3}
\item $\alpha=1$ and either $mn \nmid v$ or $v=2mn$; \label{alpha1}
\item $v=2mn/\gcd(m,n)$. \label{smallv}
\end{enumerate}
\end{theorem}
\begin{proof}
If $m \mid n$, the existence of a solution to $\hwp(v;m,n;\alpha,\beta)$ follows from Theorem~\ref{known uniform}, so we henceforth assume that $m \nmid n$.

We first consider the case $\alpha=1$. By exception~\ref{alpha1}, we have that $v=2tmn$ for some $t\geq3$ odd, therefore
$K_v^* = tK_{2mn}^* \oplus K_t[2mn]$.
By Theorem~\ref{uniform OP}, we can factorize $tK_{2mn}^*$ into $(mn-1)$ $C_{2n}$-factors. We then fill 
$K_t[2mn]$ with a solution to $\hwp(K_t[2mn]; 2m, 2n; 1,\beta-mn+1)$, which exists by
Theorem~\ref{multipartite}.

Now suppose $\alpha\geq 2$.  Let $g=\gcd(m,n)$, and write $m=m'g$ and $n=n'g$.  Since $m$ and $n$ are divisors of $v$, we may write
$v=2tgm'n'$, where $t$ is odd, and by exception~\ref{smallv}, we have that $t\geq 3$.  Decompose $K_v^* = tK_{2gm'n'}^* \oplus K_t[2gm'n']$. Also, let $\gamma=m'n'g-1 = \lfloor \frac{2m'n'g-1}{2}\rfloor$ and define a pair $(\alpha_0,\beta_0)$ as follows:
\[
(\alpha_0,\beta_0) = \left\{ 
\begin{array}{ll} (0,\gamma), & \mbox{if $\alpha \leq \gamma + 1$},\\ 
                  (\gamma,0), & \mbox{otherwise}. 
\end{array}
\right.
\]
By Theorem~\ref{uniform OP}, there is a solution to $\hwp(2m'n'g;2m,2n;\alpha_0,\beta_0)$, so we can fill $tK_{2gm'n'}^*$ with 
$\alpha_0$ $C_{2m}$-factors and $\beta_0$ $C_{2n}$-factors.
Let $(\alpha_1,\beta_1) = (\alpha,\beta) - (\alpha_0,\beta_0)$.  It remains to show that $\hwp(K_t[2m'n'g];2m,2n;\alpha_1,\beta_1)$ has a solution.   Since $\alpha_1\geq 2$, the result then follows by Theorem~\ref{multipartite}. 
\end{proof}

Theorem~\ref{main result} improves the result of Theorem~\ref{known uniform} in the case that $v \equiv 2$ (mod 4) by allowing the possibility that $\alpha$ and $\beta$ are odd.  The case that $v \equiv 0$ (mod 4) is covered by Theorem~\ref{known uniform} except if $1 \in \{\alpha,\beta\}$.  In the case $\alpha=1$, we have the following.

\begin{lemma} \label{0(mod 4) and alpha=1}
Suppose $v \equiv 0$ (mod 4), and let $m$ and $n$ be odd integers with $3 \leq m < n$, and $mn \mid v$.  Then there is a solution to $\hwp(v;2m,2n;1,\frac{v-2}{2}-1)$.
\end{lemma}

\begin{proof}
Write $v=tmn$, where $t\equiv 0$ (mod $4$).  Decompose $K_v^*=mK_{tn}^* \oplus K_{m}[tn]$.  By Theorem~\ref{uniform OP}, there is a $C_{2n}$-factorization of $mK_{tn}^*$.  By Theorem~\ref{liu}, there is a $C_{2m}$-factorization of $K_m[t]$, and hence a $C_{2m}[n]$-factorization of $K_m[tn]$.  By Theorem~\ref{C_{2m}[n] odd alpha}, we can fill one $C_{2m}[n]$-factor with a solution of $\hwp(C_{2m}[n];2m,2n;1,n-1)$, and by Theorem~\ref{C_{2m}[n] even alpha} we can the remainder with solutions of $\hwp(C_{2m}[n];2m,2n;0,n)$.
\end{proof}

Combining the results of Theorems~\ref{known uniform} and \ref{main result} together with Lemma~\ref{0(mod 4) and alpha=1} proves the following.
\begin{theorem} \label{combined theorem}
Let $m$ and $n$ be integers with $n>m \geq 2$, and let $\alpha$ and $\beta$ be positive integers.  
\begin{enumerate}
\item If $m \mid n$, then there is a solution to $\hwp(v;2m,2n;\alpha,\beta)$ if and only if $2n \mid v$ and $\alpha+\beta=\frac{v-2}{2}$.  
\item If $m \nmid n$, then there is a solution to $\hwp(v;2m,2n;\alpha,\beta)$ if and only if $2m$ and $2n$ are both divisors of $v$ and $\alpha+\beta=\frac{v-2}{2}$, except possibly when at least one of the following holds:
\begin{enumerate}
\item $\beta=1$;
\item $\beta=3$,  $v \equiv 2$ (mod 4) and $\gcd(m,n)=1$;
\item $\alpha=1$ and at least one of $m$ or $n$ is even;
\item $\alpha=1$, $m$ and $n$ are odd, and either $mn \nmid v$ or $v=2mn$;
\item $v=2mn/\gcd(m,n) \equiv 2$ (mod 4), and $\alpha$ and $\beta$ are odd.
\end{enumerate}
\end{enumerate}
\end{theorem}

\section{Acknowledgements}

The first and second authors gratefully acknowledge support from NSERC Discovery grants RGPIN-435898-2013 and RGPIN-2016-04178, respectively.

\end{document}